\documentclass[oneside,english]{amsart}
\usepackage[T1]{fontenc}
\usepackage[latin9]{inputenc}
\usepackage{amstext}
\usepackage{amsthm}
\usepackage{amssymb}

\makeatletter
\numberwithin{equation}{section}
\numberwithin{figure}{section}
\theoremstyle{plain}
\newtheorem*{thm*}{\protect\theoremname}
\theoremstyle{plain}
\newtheorem{thm}{\protect\theoremname}
\theoremstyle{definition}
\newtheorem{defn}[thm]{\protect\definitionname}
\theoremstyle{definition}
\newtheorem{example}[thm]{\protect\examplename}
\theoremstyle{plain}
\newtheorem{prop}[thm]{\protect\propositionname}
\theoremstyle{remark}
\newtheorem{rem}[thm]{\protect\remarkname}

\usepackage{xypic}
\usepackage{amsmath, amscd, amsthm, amssymb, graphics, xypic, mathrsfs, setspace, fancyhdr, times, bm, pdfsync, enumitem}
\usepackage[usenames, dvipsnames, svgnames, table]{xcolor}
\usepackage[a4paper,top=1.05in, bottom=1.05in, left=1.05in, right=1.05in]{geometry}
\usepackage[colorlinks=true,pagebackref=true]{hyperref}
\hypersetup{backref}

\makeatother

\usepackage{babel}
\providecommand{\definitionname}{Definition}
\providecommand{\examplename}{Example}
\providecommand{\propositionname}{Proposition}
\providecommand{\remarkname}{Remark}
\providecommand{\theoremname}{Theorem}

\begin{document}
\author{Adrien Dubouloz}
\address{IMB UMR5584, CNRS, Univ. Bourgogne Franche-Comté, F-21000 Dijon, France.}
\email{adrien.dubouloz@u-bourgogne.fr}
\thanks{The author thanks the organizers of the Kinosaki Algebraic Geometry
Symposium for their generous support and the very stimulating discussions
held in a warm and friendly atmosphere during this conference.}
\title{$\mathbb{A}^{1}$-cylinders over smooth $\mathbb{A}^{1}$-fibered
affine surfaces}
\begin{abstract}
We give a general structure theorem for affine $\mathbb{A}^{1}$-fibrations
on smooth quasi-projective surfaces. As an application, we show that
every smooth $\mathbb{A}^{1}$-fibered affine surface non-isomorphic
to the total space of a line bundle over a smooth affine curve fails
the Zariski Cancellation Problem. The present note is an expanded
version of a talk given at the Kinosaki Algebraic Geometry Symposium
in October 2019.
\end{abstract}

\maketitle

\section*{Introduction}

The Zariski Cancellation Problem asks under which circumstances the
existence of a biregular isomorphism between the cartesian products
$X\times\mathbb{A}^{n}$ and $Y\times\mathbb{A}^{n}$ of two algebraic
varieties $X$ and $Y$ with the affine space $\mathbb{A}^{n}$, say
over a an algebraically closed field of characteristic zero, implies
that the varieties $X$ and $Y$ are isomorphic. Cancellation is known
to hold for smooth curves \cite{AEH72} and for a large class of algebraic
varieties characterized roughly by the property that they are not
dominantly covered by images of the affine line $\mathbb{A}^{1}$
(see e.g. \cite{IiFu77}, \cite{Dry07}). A specific stronger characterization
due to Makar-Limanov \cite{BML05} asserts that if $X$ is an affine
variety which does not admit any effective algebraic action of the
additive group $\mathbb{G}_{a}$ then every isomorphism $X\times\mathbb{A}^{1}\simeq Y\times\mathbb{A}^{1}$
induces an isomorphism $X\simeq Y$ (this is no longer true for products
with affine spaces $\mathbb{A}^{n}$ of higher dimension, see e.g.
\cite{Dub15,Dub19}).

Among smooth affine surfaces with an effective action of the additive
group $\mathbb{G}_{a}$, cancellation is known to hold for the affine
plane $\mathbb{A}^{2}$ by \cite{Miy75, MS80}. The first celebrated
examples of smooth affine surfaces with effective $\mathbb{G}_{a}$-actions
which fail cancellation were constructed by Danielewski \cite{Dan89}:
he established that the smooth surfaces $S_{n}$ in $\mathbb{A}^{3}$
defined by the equations $x^{n}z=y(y-1)$, where $n\geq1$, are pairwise
non-isomorphic but that their $\mathbb{A}^{1}$-cylinders $S_{n}\times\mathbb{A}^{1}$
are all isomorphic. Since then, many other families of examples of
smooth affine surfaces with effective $\mathbb{G}_{a}$-actions which
fail cancellation have been constructed (see e.g. \cite{Ru14,FKZ18}
and the references therein for a survey). All these constructions
are derived from variants of the nowadays called ``Danielewski fiber
product trick'', which depends on the study of the structure of the
algebraic quotient morphisms of $\mathbb{G}_{a}$-actions on affine
surfaces. These quotient morphisms are surjections $\pi:S\rightarrow C$
onto smooth affine curves, with generic fiber isomorphic to the affine
line $\mathbb{A}^{1}$ over the function field of $C$, called \emph{$\mathbb{A}^{1}$-fibrations}.
The local structure of these fibrations in a neighborhood of their
degenerate fibers has been studied by many authors after the pioneering
work of Miyanishi \cite{MiyBook} and Fieseler \cite{Fie94}. The
first result presented in this note is a general structure theorem
for affine $\mathbb{A}^{1}$-fibrations on smooth quasi-projective
surfaces which generalizes and encompasses, in a different language,
all formerly known descriptions:
\begin{thm*}
Let $S$ be a smooth quasi-projective surface and let $\pi:S\rightarrow C$
be an affine $\mathbb{A}^{1}$-fibration over a smooth algebraic curve
$C$. Then there exists a smooth algebraic space $\mathcal{C}$ of
dimension $1$ endowed with a surjective quasi-finite birational morphism
of finite type $\alpha:\mathcal{C}\rightarrow C$ and a factorization
\[
\pi=\alpha\circ\rho:S\stackrel{\rho}{\rightarrow}\mathcal{C}\stackrel{\alpha}{\rightarrow}C
\]
 where $\rho:S\rightarrow\mathcal{C}$ is an \'etale locally trivial
$\mathbb{A}^{1}$-bundle.
\end{thm*}
In the case of a smooth $\mathbb{A}^{1}$-fibration $\pi:S\rightarrow C$
on a smooth affine surface $S$, the above result was already established
in \cite{Fie94} and \cite{Dub05} (see also \cite{FKZ18}) with the
additional observation that in this particular case, the algebraic
space curve $\alpha:\mathcal{C}\rightarrow C$ is a smooth scheme,
in general not separated. We will see below that for non-smooth $\mathbb{A}^{1}$-fibrations,
the existence of multiple fibers forces to consider algebraic space
curves $\alpha:\mathcal{C}\rightarrow C$ which are not schemes. This
fact was already observed in \cite{Dub15,Dub19} where the existence
of a factorization as above was established for some particular examples
of $\mathbb{A}^{1}$-fibrations with multiple fibers.\\

Our second result is an application of the above structure theorem
to the construction of smooth affine surfaces which fail cancellation.
By applying a new variant of the Danielwski fiber product trick construction,
we obtain the following characterization which basically fully settles
the Zariski Cancellation Problem for smooth affine surfaces:
\begin{thm*}
Let $S$ be a smooth affine surface and let $\pi:S\rightarrow C$
be an $\mathbb{A}^{1}$-fibration over a smooth affine curve $C$.
Then the following alternative holds:

a) If $\pi:S\rightarrow C$ is isomorphic to the structure morphism
of a line bundle over $C$ then every smooth affine surface $S'$
such that $S'\times\mathbb{A}^{1}\simeq S\times\mathbb{A}^{1}$ is
isomorphic to $S$.

b) Otherwise, there exists a smooth affine $\mathbb{A}^{1}$-fibered
surface $S'$ non-isomorphic to $S$ such that $S\times\mathbb{A}^{1}$
is isomorphic to $S'\times\mathbb{A}^{1}$.
\end{thm*}
The characterization in the above theorem strongly overlaps similar
results established in \cite{FKZ18} for smooth surfaces which admit
$\mathbb{A}^{1}$-fibrations with reduced fibers. We also refer the
reader to a forthcoming article in collaboration S. Kaliman and M.
Zaidenberg in which a similar result is established by different methods.\\

In this note, all schemes and algebraic spaces are assumed to be defined
for simplicity over the field of complex numbers $\mathbb{C}$. We
refer the reader to \cite{Knu} for the basic properties of algebraic
spaces which are used throughout the text. Some of the results in
this note are given without complete and detailed proofs, these will
appear elsewhere.

\section{Smooth $\mathbb{A}^{1}$-fibered surfaces as \'etale locally trivial
$\mathbb{A}^{1}$-bundles}
\begin{defn}
An \emph{$\mathbb{A}^{1}$-fibration} on a smooth quasi-projective
surface $S$ is a surjective affine morphism $\pi:S\rightarrow C$
to a smooth algebraic curve $C$, whose fiber over the generic point
of $C$ is isomorphic to the affine line $\mathbb{A}_{K(C)}^{1}$
over the function field $K(C)$ of $C$.

An $\mathbb{A}^{1}$-fibration on a smooth quasi-projective surface
$S$ is said to be of \emph{affine type} (resp. \emph{complete type})
if the curve $C$ is affine (resp. complete).
\end{defn}

\begin{example}
Given a smooth algebraic curve $C$, a $\mathbb{P}^{1}$-bundle $\overline{\pi}:\mathbb{P}(E)\rightarrow C$
for some vector bundle $E$ of rank $2$ on $C$ and a section $\sigma:C\rightarrow\mathbb{P}(E)$
of $\overline{\pi}$, the restriction $\pi:S=\mathbb{P}(E)\setminus\sigma(C)\rightarrow C$
of $\overline{\pi}$ to the complement of $\sigma(C)$ is an $\mathbb{A}^{1}$-fibration
which is a\emph{ Zariski locally trivial $\mathbb{A}^{1}$-bundle}
over $C$, that is, there exists a covering of $C$ by Zariski open
susbets $C_{i}$, $i\in I$, and isomorphisms $\pi^{-1}(C_{i})\simeq C_{i}\times\mathbb{A}^{1}$
of schemes over $C_{i}$ for every $i\in I$.

If $C$ is affine, such Zariski locally trivial $\mathbb{A}^{1}$-fibrations
$\pi:S\rightarrow C$ are simply line bundles. This is no longer the
case in general when $C$ is complete. For instance, let $S\subset\mathbb{A}^{3}$
be the smooth affine quadric surface with equation $xz=y^{2}-1$.
Then the morphism 
\[
\pi:S\rightarrow\mathbb{P}^{1},\quad(x,y,z)\mapsto[x:y-1]=[y+1:z]
\]
is a Zariski locally trivial $\mathbb{A}^{1}$-bundle which cannot
be a line bundle. Indeed, otherwise the zero section of this line
bundle would be a complete curve contained in $S$, which is impossible
as $S$ is affine.
\end{example}

\begin{prop}
\label{prop:No-degen-implies-A1Bun }\cite[Lemma 1.3]{KaMiy78} Let
$S$ be a smooth quasi-projective surface and let $\pi:S\rightarrow C$
be an $\mathbb{A}^{1}$-fibration onto a smooth curve $C$. Assume
that all scheme-theoretic fibers of $\pi:S\rightarrow C$ are irreducible
and reduced. Then $\pi:S\rightarrow C$ is Zariski locally trivial
$\mathbb{A}^{1}$-bundle.
\end{prop}

\begin{defn}
Let $\pi:S\rightarrow C$ be an $\mathbb{A}^{1}$-fibration on a smooth
quasi-projective surface $S$. A scheme-theoretic fiber of $\pi$
over a closed point $c$ of $C$ which is either reducible or non-reduced
is called \emph{degenerate}.
\end{defn}

By \cite[Lemma 1.4.2]{MiyBook}, every degenerate fiber of an $\mathbb{A}^{1}$-fibration
$\pi:S\rightarrow C$ is a disjoint union of curves isomorphic to
the complex affine line $\mathbb{A}^{1}$ when endowed with respective
reduced structures.
\begin{example}
Let $n\geq2$, let $P(y)=\prod_{i=1}^{r}(y-y_{i})^{m_{i}}\in\mathbb{C}[y]$
be a non-constant monic polynomial with $r\geq1$ distinct roots $y_{i}\in\mathbb{C}$
of respective multiplicities $m_{i}\geq1$ and let $S\subset\mathbb{A}^{3}$
be the affine surface defined by the equation $x^{n}z=P(y)-x$. Then
$S$ is smooth by the Jacobian criterion and the projection $\mathrm{pr}_{x}$
induces an $\mathbb{A}^{1}$-fibration $\pi:S\rightarrow\mathbb{A}^{1}$
which restricts to the trivial $\mathbb{A}^{1}$-bundle $\mathbb{A}^{1}\setminus\{0\}\times\mathrm{Spec}(\mathbb{C}[y])$
over $\mathbb{A}^{1}\setminus\{0\}=\mathrm{Spec}(\mathbb{C}[x^{\pm1}])$.
On the other hand, the scheme-theoretic fiber $\pi^{-1}(\{0\})$ decomposes
as the disjoint union of the schemes 
\[
F_{i}=\mathrm{Spec}(\mathbb{C}[y]/((y-y_{i})^{m_{i}}[z]),\quad i=1,\ldots,r
\]
whose reductions are all isomorphic to the affine line $\mathbb{A}^{1}=\mathrm{Spec}(\mathbb{C}[z])$.
\end{example}

\subsection{The smooth relatively connected quotient of an $\mathbb{A}^{1}$-fibration}

In this subsection, we show that every $\mathbb{A}^{1}$-fibration
$\pi:S\rightarrow C$ on a smooth quasi-projective surface $S$ factors
through a smooth morphism with connected fibers $\rho:S\rightarrow\mathcal{C}$
over a suitably defined algebraic space curve $\mathcal{C}$ over
$C$.
\begin{defn}
\label{def:Smooth-curve-nfold} Let $C$ be a smooth algebraic curve.
A \emph{smooth multifold algebraic space $C$-curve }is a smooth algebraic
space $\mathcal{C}$ of dimension $1$ endowed with a surjective quasi-finite
birational morphism of finite type $\alpha:\mathcal{C}\rightarrow C$
such that $\alpha_{*}\mathcal{O}_{\mathcal{C}}=\mathcal{O}_{C}$.
\end{defn}

By generic smoothness, there exists a non empty maximal Zariski open
subset $U$ of $C$ over which $\alpha:\mathcal{C}\rightarrow C$
restricts to an \'etale morphism $\alpha:\alpha^{-1}(U)\rightarrow U$.
Given any separated open subset $\mathcal{V}\subset\alpha^{-1}(U)$,
the restriction $\alpha|_{\mathcal{V}}:\mathcal{V}\rightarrow C$
is a separated birational quasi-finite \'etale morphism. Since a
quasi-finite morphism is quasi-affine, $\mathcal{V}$ is thus a quasi-projective
scheme and $\alpha|_{\mathcal{V}}:\mathcal{V}\rightarrow C$ is an
open immersion by Zariski main theorem \cite[Théorème 8.12.6]{EGAIV-3}.
It follows in particular that there exists finitely many points $c_{1},\ldots,c_{s}$
of $C$ over which $\alpha:\mathcal{C}\rightarrow C$ is not an isomorphism.
Furthermore, for every $i=1,\ldots,s$, the fiber $\alpha^{-1}(c_{i})$
consists of finitely many points $\breve{c}_{i,1},\ldots,\breve{c}_{i,r_{i}}$,
and if $\alpha$ is unramified at $\breve{c}_{i,j}$ then there exists
a separated Zariski open neighborhood $\breve{U}_{i,j}$ of $\breve{c}_{i,j}$
in $\mathcal{C}$ such that $\alpha|_{\breve{U}_{i,j}}:\breve{U}_{i,j}\rightarrow C$
is an isomorphism onto its image. So one can picture a smooth multifold
algebraic space $C$-curve $\alpha:\mathcal{C}\rightarrow C$ as being
obtained from the curve $C$ by ``replacing'' finitely many points
$c_{1},\ldots,c_{s}$ of $C$ by a collection of finitely many distinct
algebraic space curve points $\breve{c}_{i,1},\ldots,\breve{c}_{i,r_{i}}$.
\begin{example}
\label{exa:Local-multi-origin}Let $C$ be smooth affine curve, let
$c_{0}\in C$ be a closed point and let $\mathcal{C}$ be the curve
obtained by gluing $r\geq2$ copies $\alpha_{i}:\mathcal{C}_{i}\stackrel{\simeq}{\rightarrow}C$,
$i=1,\ldots,r$, of $C$ by the identity outside the points $c_{i}=\alpha_{i}^{-1}(c_{0})$.
The curve $\mathcal{C}$ is a non-separated scheme on which the morphisms
$\alpha_{i}$ glue to a surjective quasi-finite birational morphism
$\alpha:\mathcal{C}\rightarrow C$ which coincides with the canonical
affinization morphism $\mathcal{C}\rightarrow\mathrm{Spec}(\Gamma(\mathcal{C},\mathcal{O}_{\mathcal{C}}))$.
The restriction of $\alpha$ over $C\setminus\{c_{0}\}$ is an isomorphism
whereas $\alpha^{-1}(c_{0})$ consists of $r$ distinct point $c_{i}$,
$i=1,\ldots,r$.
\end{example}

\begin{example}
\label{exa:Local-Alg-Space}Let $C$ be a smooth algebraic curve,
let $c_{0}\in C$ be a closed point and let $\varphi:\tilde{C}\rightarrow C$
be a quasi-finite morphism of degree $m\geq2$, with branching index
$m$ at $c$ and \'etale elsewhere. Let $\tilde{c}_{0}=\varphi^{-1}(c_{0})$
and let $\mathcal{C}$ be the smooth algebraic space obtained from
$\tilde{C}$ by identifying two points $\tilde{c},\tilde{c}'\in\tilde{C}\setminus\{\tilde{c}_{0}\}$
if $\varphi(\tilde{c})=\varphi(\tilde{c}')$. More rigorously, letting
$C_{*}=C\setminus\{c\}$ and $\tilde{C}_{*}=\tilde{C}\setminus\{\tilde{c}_{0}\}=\varphi^{-1}(C_{*})$,
$\mathcal{C}$ is the quotient of $\tilde{C}$ by the \'etale equivalence
relation 
\[
R=\mathrm{Diag}\sqcup j:\tilde{C}\sqcup\tilde{C}_{*}\times_{C_{*}}\tilde{C}_{*}\longrightarrow\tilde{C}\times_{C}\tilde{C},
\]
where $\mathrm{Diag}:\tilde{C}\rightarrow\tilde{C}\times_{C}\tilde{C}$
is the diagonal morphism and $j:\tilde{C}_{*}\times_{C_{*}}\tilde{C}_{*}\rightarrow\tilde{C}\times_{C}\tilde{C}$
is the natural open immersion. The $R$-invariant morphism $\varphi:\tilde{C}\rightarrow C$
descends through the quotient morphism $q:\tilde{C}\rightarrow\mathcal{C}=\tilde{C}/R$
to a bijective quasi-finite morphism $\alpha:\mathcal{C}\rightarrow C$
restricting to an isomorphism over $C_{*}$. On the other hand, the
inverse image of $c_{0}$ by $\alpha$ consists of a unique point
$\mathfrak{c}_{0}=q(\tilde{c}_{0})$, at which $\alpha$ ramification
index $m$. The sheaf $\alpha_{*}\mathcal{O}_{\mathcal{C}}$ is equal
to the $\mathcal{O}_{C}$-submodule of $\varphi_{*}\mathcal{O}_{\tilde{C}}$
consisting of germs of $R$-invariant regular functions on $\tilde{C}$,
hence is equal to $\mathcal{O}_{C}$. Since $R$ is not a locally
closed immmersion in a neighborhood of the point $\tilde{c}_{0}\in\tilde{C}$,
it follows that the algebraic space $\mathcal{C}$ is not locally
separated in a neighborhood of the point $\mathfrak{c}_{0}$, hence
is not a scheme.
\end{example}

\begin{thm}
\label{thm:Main-Struct}Let $\pi:S\rightarrow C$ be an $\mathbb{A}^{1}$-fibration
on a smooth quasi-projective surface. Then there exists a smooth multifold
algebraic space $C$-curve $\alpha:\mathcal{C}\rightarrow C$ unique
up to $C$-isomorphism and a smooth affine morphism with connected
fibers $\rho:S\rightarrow\mathcal{C}$ such that $\pi=\alpha\circ\rho$.
\end{thm}

\begin{proof}[Sketch of proof]
 A smooth multifold algebraic space $C$-curve $\alpha:\mathcal{C}\rightarrow C$
with the desired properties is obtained as follows. Let $c_{1},\ldots,c_{s}$
be the points over which the fibers of $\pi:S\rightarrow C$ are degenerate
and let $\breve{\alpha}:\breve{C}\rightarrow C$ be the scheme obtained
from $C$ as in Example \ref{exa:Local-multi-origin} by replacing
each point $c_{s}$ by distinct scheme points $\breve{c}_{i,1},\ldots,\breve{c}_{i,r_{i}}$,
one for each connected component $F_{i,j}$ of the fiber $\pi^{-1}(c_{i})$,
$i=1,\ldots,s$. The unique morphism $\breve{\rho}:S\rightarrow\breve{C}$
defined by 
\[
\breve{\rho}(s)=\begin{cases}
\breve{\alpha}^{-1}(\pi(s)) & \textrm{if }\pi^{-1}(\pi(s))\textrm{ is connected}\\
\breve{c}_{i,j} & \textrm{if }\text{\ensuremath{\pi(s)=c_{i}\textrm{ and }s\in F_{i,j}} }
\end{cases}
\]
is affine with connected fibers and satisfies $\pi=\breve{\alpha}\circ\breve{\rho}$.
Since $S$ is smooth and every connected component of a fiber of $\pi$
is irreducible and smooth when equipped with its reduced structure,
we see that $\breve{\rho}:S\rightarrow\breve{C}$ is smooth over a
point $\breve{c}$ of $\breve{C}$ if and only if $\breve{\rho}^{-1}(\breve{c})$
is a reduced irreducible component of $\pi^{-1}(\breve{\alpha}(\breve{c}))$.

Let $\breve{c}_{0}\in\breve{C}$ be a point such that $\breve{\rho}^{-1}(\breve{c}_{0})=mF$,
where $F\simeq\mathbb{A}^{1}$, is multiple, of multiplicity $m\geq2$
and let $s$ be a closed point of $S$ supported on $F$. Since $S$
and $F$ are smooth at $s$, there exists a germ of smooth curve $\tilde{C}_{0}\hookrightarrow S$
intersecting $F$ transversally at $s$. The restriction $\varphi=\breve{\rho}|_{\tilde{C}_{0}}:\tilde{C}_{0}\rightarrow\breve{C}$
is quasi-finite onto its image. By shrinking $\tilde{C}_{0}$ if necessary
we can assume without loss generality that the image of $\varphi$
is an affine open neighborhood $C_{0}$ of $\breve{c}_{0}$ in $\breve{C}$
with the property that $\pi^{-1}(\breve{\alpha}(\breve{c}_{0}))$
is the unique degenerate fiber of $\pi$ over $\breve{\alpha}(C_{0})$
and that $\varphi:\tilde{C}_{0}\rightarrow C_{0}$ is a quasi-finite
morphism of degree $m\geq2$, with ramification index $m$ at $s$
and \'etale elsewhere. Let $\beta_{\breve{c}_{0}}:\tilde{C}_{0}/R\rightarrow C_{0}$
be the algebraic space curve over $C_{0}$ determined by $\varphi:\tilde{C}_{0}\rightarrow C_{0}$
as in Example \ref{exa:Local-Alg-Space} and let $\alpha_{\breve{c}_{0}}:\mathcal{C}_{\breve{c}_{0}}\rightarrow\breve{C}$
be the algebraic space curve over $\breve{C}$ obtained by gluing
$\breve{C}\setminus\{\breve{c}_{0}\}$ with $\tilde{C}_{0}/R$ with
by the identity along the open subsets $C_{0}\setminus\{\breve{c}_{0}\}$
and $\beta_{\breve{c}_{0}}^{-1}(C_{0}\setminus\{\breve{c}_{0}\})\simeq C_{0}\setminus\{\breve{c}_{0}\}$
of $\breve{C}\setminus\{\breve{c}_{0}\}$ and $\tilde{C}_{0}/R$ respectively.
Letting $\mathfrak{c}_{0}=\alpha_{\breve{c}_{0}}^{-1}(\breve{c}_{0})$,
one checks locally on an \'etale cover of $S$ that $\breve{\rho}:S\rightarrow\breve{C}$
factors through an affine morphism $\rho_{\breve{c}_{0}}:S\rightarrow\mathcal{C}_{\breve{c}_{0}}$
smooth over a Zariski open neighborhood of $\mathfrak{c}_{0}$ and
such that $\rho_{\breve{c}_{0}}^{-1}(\mathfrak{c}_{0})=F$. By repeating
the above construction for each of the finitely many points of $\breve{C}$
over which the fiber of $\breve{\rho}$ is multiple, we obtain a smooth
multifold algebraic space $C$-curve $\alpha:\mathcal{C}\rightarrow C$
and a smooth affine morphism with connected fibers $\rho:S\rightarrow\mathcal{C}$
factoring $\pi$.

By construction, for every closed point $c\in C$, the fibers of $\rho:S\rightarrow\mathcal{C}$
over the points in $\alpha^{-1}(c)$ are in one-to-one correspondence
with the connected components of the fiber of $\pi:S\rightarrow C$
over $c$. Therefore, if $\alpha':\mathcal{C}'\rightarrow C$ is another
smooth multifold algebraic space $C$-curve with the same properties
then its associated smooth morphism $\rho':S\rightarrow\mathcal{C}'$
is locally constant on the fibers of $\pi:S\rightarrow C$ hence constant
on the fibers of $\rho:S\rightarrow\mathcal{C}$. This implies in
turn by faithfully flat descent that there exists a unique morphism
$\varphi:\mathcal{C}\rightarrow\mathcal{C}'$ of algebraic spaces
over $C$ such that $\rho'=\varphi\circ\rho$. Reversing the roles
of $\mathcal{C}$ and $\mathcal{C}'$, we conclude that $\varphi$
is a $C$-isomorphism.
\end{proof}
\begin{defn}
Let $\pi:S\rightarrow C$ be an $\mathbb{A}^{1}$-fibration on a smooth
quasi-projective surface $S$. The smooth multifold algebraic space
$C$-curve $\alpha:\mathcal{C}\rightarrow C$ such that $\pi$ factors
as $\pi=\alpha\circ\rho$ for some smooth affine morphism $\rho:S\rightarrow\mathcal{C}$
with connected fibers is called the \emph{smooth relatively connected
quotient} of $\pi:S\rightarrow C$.

It follows from the proof of Theorem \ref{thm:Main-Struct} that the
isomorphism type of $\alpha:\mathcal{C}\rightarrow C$ as a space
over $C$ depends only on the irreducible components, taken with their
respective multiplicities of the scheme-theoretic degenerate fibers
of $\pi:S\rightarrow C$.
\end{defn}

\subsection{\label{subsec:Toy-Example}Illustration on a toy local model}

Let $C$ be the spectrum of a discrete valuation ring $\mathcal{O}$
with maximal ideal $\mathfrak{m}$ and residue field $\mathbb{C}$
and let $t\in\mathfrak{m}$ be a uniformizing parameter. Given integers
$n,m\geq2$, let $S_{n,m}$ be the smooth affine surface in $C\times\mathbb{A}^{2}$
defined by the equation $t^{n}z=y^{m}-t$. The restriction to $S_{n,m}$
of the projection $\mathrm{pr}_{C}$ is an $\mathbb{A}^{1}$-fibration
$\pi:S_{n,m}\rightarrow C$ whose fiber over the closed point $c$
of $C$ is irreducible of multiplicity $m$, isomorphic to $\mathrm{Spec}(\mathbb{C}[y]/(y^{m})[z])$.
The curve
\[
\tilde{C}=\{z=0\}\simeq\mathrm{Spec}(\mathcal{O}[y]/(y^{m}-t))
\]
on $S$ is smooth and the restriction of $\pi$ to $\tilde{C}$ is
a finite Galois cover $\varphi:\tilde{C}\rightarrow C$ totally ramified
over the closed point of $C$, with Galois group equal to the group
$\mu_{m}$ of complex $m$-th roots of unity, The normalization $\tilde{S}_{n,m}$
of the fiber product $S_{n,m}\times_{C}\tilde{C}$ is isomorphic to
the smooth affine surface in $\tilde{C}\times\mathbb{A}^{2}$ defined
by the equation $y^{(n-1)m}z=u^{m}-1$. The Galois group $\mu_{m}$
acts on $\tilde{S}_{n,m}$ by $(y,z,u)\mapsto(\varepsilon y,\varepsilon^{-1}u,z)$,
where $\varepsilon$ is a primitive $m$-th root of unity, and the
quotient morphism $\tilde{S}_{n,m}\rightarrow\tilde{S}_{n,m}/\mu_{m}\simeq S_{n,m}$
is \'etale. The restriction to $\tilde{S}_{n,m}$ of the projection
$\mathrm{pr}_{\tilde{C}}$ is an $\mathbb{A}^{1}$-fibration $\tilde{\pi}:\tilde{S}_{n,m}\rightarrow\tilde{C}$
whose fiber over the closed point $\tilde{c}$ of $\tilde{C}$ is
reduced, consisting of $m$ disjoint copies $\tilde{F}_{i}=\{y=u-\varepsilon^{i}=0\}$,
$i=0,\ldots,m-1$, of the affine line $\mathbb{A}^{1}=\mathrm{Spec}(\mathbb{C}[z])$,
which form a unique orbit of the action of $\mu_{m}$ on $\tilde{S}_{n,m}$.

Let $\tilde{\alpha}:\tilde{\mathcal{C}}\rightarrow\tilde{C}$ be the
scheme obtained as in Example \ref{exa:Local-multi-origin} by gluing
$m$ copies $\tilde{\alpha}_{i}:\tilde{\mathcal{C}}_{i}\stackrel{\simeq}{\rightarrow}\tilde{C}$
of $\tilde{C}$ by the identity outside the points $\tilde{c}_{i}=\tilde{\alpha}_{i}^{-1}(\tilde{c})$
and let $\tilde{\rho}:\tilde{S}_{n,m}\rightarrow\tilde{\mathcal{C}}$
be the unique morphism lifting $\tilde{\pi}:\tilde{S}_{n,m}\rightarrow\tilde{C}$
and mapping $\tilde{F}_{i}$ to $\tilde{c}_{i}$. For every $i=0,\ldots,m-1$,
the open subset $\tilde{S}_{n,m}^{i}=\tilde{S}_{n,m}\setminus\bigcup_{j\neq i}\tilde{F}_{j}$
of $\tilde{S}_{n,m}$ is isomorphic to $\tilde{C}\times\mathrm{Spec}(\mathbb{C}[v_{i}])$
where $v_{i}$ is the regular function on $\tilde{S}_{n,m}^{i}$ defined
as the restriction of the rational function 
\[
v_{i}=\frac{(u-\varepsilon^{i})}{y^{(n-1)m}}=\frac{z}{\prod_{j\neq i}(u-\varepsilon^{j})}
\]
on $\tilde{S}_{n,m}$. Via the so-defined isomorphism $\tilde{S}_{n,m}^{i}\simeq\tilde{C}\times\mathbb{A}^{1}$,
the restriction of $\tilde{\rho}:\tilde{S}_{n,m}\rightarrow\tilde{\mathcal{C}}$
on $\tilde{S}_{n,m}^{i}$ coincides with the composition of the projection
$\mathrm{pr}_{\tilde{C}}$ with the inclusion of $\tilde{C}$ as the
open subset $\tilde{\mathcal{C}}_{i}$ of $\tilde{\mathcal{C}}$.
It follows that $\tilde{\rho}:\tilde{S}_{n,m}\rightarrow\tilde{\mathcal{C}}$
is a Zariski locally trivial $\mathbb{A}^{1}$-bundle, in particular,
a smooth morphism with connected fibers.

The action of the Galois group $\mu_{m}$ on $\tilde{S}_{n,m}$ descends
to a free $\mu_{m}$-action on $\tilde{\mathcal{C}}$ defined by 
\[
\tilde{\mathcal{C}}_{i}\ni y\mapsto\varepsilon y\in\tilde{\mathcal{C}}_{i+1\textrm{ mod }m}.
\]
The geometric quotient of $\tilde{\mathcal{C}}$ by this $\mu_{m}$-action
is an algebraic space $\mathcal{C}$, isomorphic to the algebraic
space $\tilde{C}/R$ associated to the finite cover $\varphi:\tilde{C}\rightarrow C$
as in Example \ref{exa:Local-Alg-Space}, and the quotient morphism
$q:\tilde{\mathcal{C}}\rightarrow\mathcal{C}=\tilde{\mathcal{C}}/\mu_{m}$
is an \'etale $\mu_{m}$-torsor. Let 
\[
\rho:S_{n,m}\simeq\tilde{S}_{n,m}/\mu_{m}\rightarrow\tilde{\mathcal{C}}/\mu_{m}=\mathcal{C}\;\textrm{and}\;\alpha:\mathcal{C}=\tilde{\mathcal{C}}/\mu_{m}\rightarrow\tilde{C}/\mu_{m}\simeq C
\]
be the morphisms induced by the $\mu_{m}$-equivariant morphisms $\tilde{\rho}:\tilde{S}_{n,m}\rightarrow\tilde{\mathcal{C}}$
and $\tilde{\alpha}:\tilde{\mathcal{C}}\rightarrow\tilde{C}$ respectively.
The inverse image of $c$ by $\alpha$ consists of a unique point
$\mathfrak{c}$ which is the image by $q:\tilde{\mathcal{C}}\rightarrow\mathcal{C}$
of the $\mu_{m}$-orbit formed by the points $\tilde{c}_{i}$, $i=0,\ldots,m-1$,
and we have a commutative diagram \[\xymatrix{\tilde{S}_{n,m} \ar[d]_{\tilde{\rho}} \ar[r] & S_{n,m}\simeq \tilde{S}_{n,m}/\mu_m \ar[d]^{\rho} \\ \tilde{\mathcal{C}} \ar[r]^{q} \ar[d]_{\tilde{\alpha}} & \mathcal{C}=\tilde{\mathcal{C}}/\mu_m \ar[d]^{\alpha} \\ \tilde{C} \ar[r] & C\simeq \tilde{C}/\mu_m}\]
in which the top square is cartesian. Since $q:\tilde{\mathcal{C}}\rightarrow\mathcal{C}$
is \'etale and $\tilde{\rho}:\tilde{S}_{n,m}\rightarrow\tilde{\mathcal{C}}$
is smooth, the morphism $\rho:S_{n,m}\rightarrow\mathcal{C}$ is thus
smooth, with fiber over $\mathfrak{c}$ equal to $\pi^{-1}(c)$ endowed
with its reduced structure.

\subsection{\label{subsec:L-torsors}$\mathbb{A}^{1}$-fibrations as torsors
under \'etale locally trivial line bundles}

Let $\pi:S\rightarrow C$ be an $\mathbb{A}^{1}$-fibration on a smooth
quasi-projetive surface, let $\alpha:\mathcal{C}\rightarrow C$ be
its smooth relatively connected quotient and let $\rho:S\rightarrow\mathcal{C}$
be the corresponding smooth affine morphism with connected fibers.
For every \'etale morphism $f:B\rightarrow\mathcal{C}$ from a smooth
algebraic curve $B$, the projection $\mathrm{pr}_{B}:S\times_{\mathcal{C}}B\rightarrow B$
is a smooth $\mathbb{A}^{1}$-fibration with connected fibers, hence
is a Zariski locally trivial $\mathbb{A}^{1}$-bundle by virtue of
Proposition \ref{prop:No-degen-implies-A1Bun }. It follows that $\rho:S\rightarrow\mathcal{C}$
is an \'etale locally trivial $\mathbb{A}^{1}$-bundle over $\mathcal{C}$.
Summing-up, we obtain:
\begin{thm}
\label{thm:A1Fib-Torsor}Every $\mathbb{A}^{1}$-fibration $\pi:S\rightarrow C$
on a smooth quasi-projective surface $S$ decomposes as an \'etale
locally trivial $\mathbb{A}^{1}$-bundle $\rho:S\rightarrow\mathcal{C}$
over a smooth multifold algebraic $C$-space curve $\mathcal{C}$
followed by the structure morphism $\alpha:\mathcal{C}\rightarrow C$
of $\mathcal{C}$.
\end{thm}

Since the automorphism group of the affine line $\mathbb{A}^{1}$
is the affine group $\mathrm{Aff}_{1}=\mathbb{G}_{m}\ltimes\mathbb{G}_{a}$,
every \'etale locally trivial $\mathbb{A}^{1}$-bundle $\rho:S\rightarrow\mathcal{C}$
is an \'etale affine-linear bundle, isomorphic to the associated
fiber bundle $V\times^{\mathrm{Aff}_{1}}\mathbb{A}^{1}\rightarrow\mathcal{C}$
of a principal homogeneous $\mathrm{Aff}_{1}$-bundle $V\rightarrow\mathcal{C}$
over $\mathcal{C}$. It follows that there exists a uniquely determined
\'etale locally trivial line bundle $p:L_{S/\mathcal{C}}\rightarrow\mathcal{C}$,
considered as a locally constant group scheme over $\mathcal{C}$
for the group law induced by the addition of germs of sections, such
that $\rho:S\rightarrow\mathcal{C}$ can be further equipped with
the structure of an \'etale $L_{S/\mathcal{C}}$-torsor, that is,
a principal homogeneous bundle under the action of $L_{S/\mathcal{C}}$.
Namely, the class of $L_{S/\mathcal{C}}$ in the Picard group $\mathrm{Pic}(\mathcal{C})=H_{\mathrm{\acute{e}t}}^{1}(\mathcal{C},\mathbb{G}_{m})$
of $\mathcal{C}$ coincides with the image of the isomorphism class
of $V\rightarrow\mathcal{C}$ in $H_{\mathrm{\acute{e}t}}^{1}(\mathcal{C},\mathrm{Aff}_{1})$
by the map $H_{\mathrm{\acute{e}t}}^{1}(\mathcal{C},\mathrm{Aff}_{1})\rightarrow H_{\mathrm{\acute{e}t}}^{1}(\mathcal{C},\mathbb{G}_{m})$
in the long exact sequence of non-abelian cohomology

\[
\cdots\rightarrow H^{0}\left(\mathcal{C},\mathbb{G}_{m}\right)\rightarrow H_{\mathrm{\acute{e}t}}^{1}\left(\mathcal{C},\mathbb{G}_{a}\right)\rightarrow H_{\mathrm{\acute{e}t}}^{1}\left(\mathcal{C},\mathrm{Aff}_{1}\right)\rightarrow H_{\mathrm{\acute{e}t}}^{1}\left(\mathcal{C},\mathbb{G}_{m}\right)
\]
associated to the short exact sequence $0\rightarrow\mathbb{G}_{a}\rightarrow\mathrm{Aff}_{1}\rightarrow\mathbb{G}_{m}\rightarrow0$
of \'etale sheaves of groups on $\mathcal{C}$. Isomorphism classes
of \'etale torsors under a given \'etale locally trivial line bundle
$p:L\rightarrow\mathcal{C}$ are in turn classified by the cohomology
group $H_{\mathrm{\acute{e}t}}^{1}(\mathcal{C},L)$ (see e.g. \cite[XI.4]{SGA1}
or \cite[III.2.4]{Gir}).

The line bundle $p:L_{S/\mathcal{C}}\rightarrow\mathcal{C}$ associated
to a given \'etale locally trivial $\mathbb{A}^{1}$-bundle $\rho:S\rightarrow\mathcal{C}$
can be alternatively described as follows: since $\rho:S\rightarrow\mathcal{C}$
is an \'etale locally trivial $\mathbb{A}^{1}$-bundle, the pull-back
homomorphism 
\[
\rho^{*}:H_{\mathrm{\acute{e}t}}^{1}(\mathcal{C},\mathbb{G}_{m})\rightarrow H_{\mathrm{\acute{e}t}}^{1}(S,\mathbb{G}_{m})\simeq\mathrm{Pic}(S)
\]
is an isomorphism. The same local description as in \cite[16.4.7]{EGAIV-4}
then implies that $L_{S/\mathcal{C}}$ is the unique \'etale line
bundle on $\mathcal{C}$ whose image by $\rho^{*}$ is equal to the
relative tangent line bundle $T_{S/\mathcal{C}}=\mathrm{Spec}(\mathrm{Sym}^{\cdot}\Omega_{S/\mathcal{C}})\rightarrow S$
of $\rho:S\rightarrow\mathcal{C}$, where $\Omega_{S/\mathcal{C}}$
denotes the sheaf of relative Kähler differentials of $S$ over $\mathcal{C}$.
\begin{example}
\label{exa:Mult2-fiber}Let $Q\subset\mathbb{P}^{2}$ be a smooth
plane conic. The pencil $\mathbb{P}^{2}\dashrightarrow\mathbb{P}^{1}$
generated by $Q$ and twice its projective tangent line $L_{q}$ at
a given point $q\in Q$ restricts on $S_{1}=\mathbb{P}^{2}\setminus Q$
to an $\mathbb{A}^{1}$-fibration $\pi_{1}:S_{1}\rightarrow\mathbb{A}^{1}$
with irreducible fibers, having 
\[
\pi_{1}^{-1}(0)=(2L_{q}\cap S_{1})\simeq2\mathbb{A}^{1}
\]
as a unique degenerate fiber of multiplicity $2$. Similarly, given
any interger $n\geq2$, the smooth surface $S_{n,2}\subset\mathbb{A}^{3}$
defined by the equation $x^{n}z=y^{2}-x$ has an $\mathbb{A}^{1}$-fibration
$\pi_{n,2}:S_{2}\rightarrow\mathbb{A}^{1}$ induced by the projection
$\mathrm{pr}_{x}$ whose unique degenerate fiber $\pi_{n,2}^{-1}(0)\simeq\mathrm{Spec}(\mathbb{C}[y]/(y^{2})[z])$
is irreducible of multiplicity $2$.

By Theorem \ref{thm:A1Fib-Torsor}, $\pi_{1}$ and $\pi_{2}$ factors
through \'etale locally trivial $\mathbb{A}^{1}$-bundles over the
multifold algebraic space curve $\alpha:\mathcal{C}\rightarrow\mathbb{A}^{1}$
defined as the quotient of $\tilde{\mathbb{A}}^{1}=\mathrm{Spec}(\mathbb{C}[u])$
by the equivalence relation 
\[
\mathbb{A}^{1}\setminus\{0\}\ni u\sim-u\in\mathbb{A}^{1}\setminus\{0\}.
\]
Equivalently, $\mathcal{C}$ is the quotient of the affine line with
a double origin $\tilde{\alpha}:\tilde{\mathcal{C}}\rightarrow\tilde{\mathbb{A}}^{1}$
obtained by gluing two copies $\tilde{\mathcal{C}}_{\pm}$ of $\tilde{\mathbb{A}}^{1}$
by the identity outside their respective origins by the free $\mu_{2}$-action
defined by $\tilde{\mathcal{C}}_{\pm}\ni u\mapsto-u\in\tilde{\mathcal{C}}_{\mp}$.
\'Etale descent along the quotient $\mu_{2}$-torsor $q:\tilde{\mathcal{C}}\rightarrow\tilde{\mathcal{C}}/\mu_{2}\simeq\mathcal{C}$
induces a one-to-one correspondence between isomorphism classes of
\'etale locally trivial line bundles on $\mathcal{C}$ and isomorphism
classes of $\mu_{2}$-linearized line bundles on $\tilde{\mathcal{C}}$.
Since the two origins of $\tilde{\mathcal{C}}$ form a unique $\mu_{2}$-orbit,
every $\mu_{2}$-linearized line bundle on $\tilde{\mathcal{C}}$
is the pull-back by $\tilde{\alpha}:\tilde{\mathcal{C}}\rightarrow\mathbb{A}^{1}$
of a $\mu_{2}$-linearlized line bundle on $\mathbb{A}^{1}$ for the
$\mu_{2}$-action $u\mapsto-u$. It follows that $\mathrm{Pic}(\mathcal{C})$
is isomorphic to $\mathbb{Z}_{2}$, generated by the class of the
\'etale locally trivial line bundle $p:L\rightarrow\mathcal{C}$
corresponding to the trivial line bundle $\tilde{\mathcal{C}}\times\mathrm{Spec}(\mathbb{C}[\ell])$
endowed with the non-trivial $\mu_{2}$-linearization given by $\ell\mapsto-\ell$
on the second factor. Noting that this line bundle on $\tilde{\mathcal{C}}$
is isomorphic to the cotangent line bundle $T_{\tilde{\mathcal{C}}}^{\vee}$
of $\tilde{\mathcal{C}}$ endowed with its canonical $\mu_{2}$-linearization,
we see that $L$ is isomorphic to the cotangent line bundle $T_{\mathcal{C}}^{\vee}\rightarrow\mathcal{C}$
of $\mathcal{C}$.

Let $\rho_{1}:S_{1}\rightarrow\mathcal{C}$ and $\rho_{n,2}:S_{n,2}\rightarrow\mathcal{C}$
be the \'etale locally trivial $\mathbb{A}^{1}$-bundles factoring
$\pi_{1}$ and $\pi_{n,2}$ respectively. The Picard group of $S_{1}=\mathbb{P}^{2}\setminus Q$
is isomorphic to $\mathbb{Z}_{2}$ generated by the restriction of
$\mathcal{O}_{\mathbb{P}^{1}}(1)$, which is equal to canonical line
bundle $\Lambda^{2}T_{S_{1}}^{\vee}$ of $S_{1}$. The Picard group
of $S_{n,2}$ is also isomorphic to $\mathbb{Z}_{2}$, generated for
instance by the line bundle corresponding to the Cartier divisor $D=\{x=y=0\}$
on $S_{n,2}$, but in contrast, since $S_{n,2}$ is a smooth hypersurface
in $\mathbb{A}^{3}$, its canonical bundle $\Lambda^{2}T_{S_{n,2}}^{\vee}$
is trivial by adjunction formula. Since the homomorphisms $\rho_{1}^{*}:\mathrm{Pic}(\mathcal{C})\rightarrow\mathrm{Pic}(S_{1})$
and $\rho_{n,2}^{*}:\mathrm{Pic}(\mathcal{C})\rightarrow\mathrm{Pic}(S_{n,2})$
are isomorphisms, we then deduce from the cotangent exact sequences
\[
0\rightarrow\rho^{*}T_{\mathcal{C}}^{\vee}\rightarrow T_{S}^{\vee}\rightarrow T_{S/\mathcal{C}}^{\vee}\rightarrow0
\]
of the morphisms $\rho_{1}:S_{1}\rightarrow\mathcal{C}$ and $\rho_{n,2}:S_{n,2}\rightarrow\mathcal{C}$
respectively that 
\[
T_{S_{1}/\mathcal{C}}^{\vee}\simeq(\rho_{1}^{*}T_{\mathcal{C}}^{\vee})^{\vee}\otimes\Lambda^{2}T_{S_{1}}^{\vee}\simeq\rho_{1}^{*}T_{\mathcal{C}}\otimes\rho_{1}^{*}T_{\mathcal{C}}^{\vee}
\]
is the trivial line bundle on $S_{1}$and that the line bundle $T_{S_{n,2}/\mathcal{C}}^{\vee}$
is isomorphic to $\rho_{n,2}^{*}T_{\mathcal{C}}$.

This implies in turn that $\rho_{1}:S_{1}\rightarrow\mathcal{C}$
is an \'etale locally trivial torsor under the trivial line bundle
on $\mathcal{C}$, in other words, a principal homogeneous $\mathbb{G}_{a,\mathcal{C}}$-bundle,
whereas $\rho_{n,2}:S_{n,2}\rightarrow\mathcal{C}$ is an \'etale
locally trivial torsor under the tangent line bundle $T_{\mathcal{C}}$
of $\mathcal{C}$. One can check further on a suitable \'etale cover
of $\mathcal{C}$ that the isomorphism classes in $H_{\mathrm{\acute{e}t}}^{1}(\mathcal{C},T_{\mathcal{C}})$
of the $T_{\mathcal{C}}$-torsors $\rho_{n,2}:S_{n,2}\rightarrow\mathcal{C}$,
$n\geq2$, are pairwise distinct.
\end{example}

\section{$\mathbb{A}^{1}$-cylinders of smooth $\mathbb{A}^{1}$-fibered surfaces
of affine type}

Given a smooth quasi-projective $\mathbb{A}^{1}$-fibered surface
$\pi:S\rightarrow C$ with smooth relatively connected quotient $\alpha:\mathcal{C}\rightarrow C$,
the isomorphism class of $S$ as a scheme over $\mathcal{C}$ is determined
by the pair consisting of the isomorphism class in $\mathrm{Pic}(\mathcal{C})$
of the \'etale line bundle $p:L_{S/\mathcal{C}}\rightarrow\mathcal{C}$
under which the associated \'etale locally trivial $\mathbb{A}^{1}$-bundle
$\rho:S\rightarrow\mathcal{C}$ is an $L_{S/\mathcal{C}}$-torsor
and of the isomorphism class of this torsor in $H_{\mathrm{\acute{e}t}}^{1}(\mathcal{C},L_{S/\mathcal{C}})$.
In contrast, the following results shows in particular that when $C$
is an affine curve, the isomorphism class as scheme over $\mathcal{C}$
of the cylinder $S\times\mathbb{A}^{1}$ over $S$ is independent
of the class of $\rho:S\rightarrow\mathcal{C}$ in $H_{\mathrm{\acute{e}t}}^{1}(\mathcal{C},L_{S/\mathcal{C}})$.
\begin{thm}
\label{thm:Cylinders}Let $\pi:S\rightarrow C$ and $\pi':S'\rightarrow C$
be smooth quasi-projective $\mathbb{A}^{1}$-fibered surfaces over
a same \underline{affine} curve $C$. Let $\alpha:\mathcal{C}\rightarrow C$
and $\alpha':\mathcal{C}'\rightarrow C$ be their respective smooth
relatively connected quotients and let $\rho:S\rightarrow\mathcal{C}'$
and $\rho':S'\rightarrow\mathcal{C}'$ be the associated \'etale
$L_{S/\mathcal{C}}$-torsor and $L_{S'/\mathcal{C}'}$-torsor respectively.

Then the threefolds $S\times\mathbb{A}^{1}$ and $S'\times\mathbb{A}^{1}$
are isomorphic as schemes over $C$ if and only if there exists a
$C$-isomorphism $\psi:\mathcal{C}\rightarrow\mathcal{C}'$ such that
$\psi^{*}L_{S'/\mathcal{C}'}\simeq L_{S/\mathcal{C}}$.
\end{thm}

\begin{proof}
A $C$-isomorphism $\Psi:S\times\mathbb{A}^{1}\stackrel{\simeq}{\rightarrow}S'\times\mathbb{A}^{1}$
induces for each closed point $c\in C$ a multiplicity preserving
one-to-one correspondence between the irreducible components of the
fiber of $\pi\circ\mathrm{pr}_{S}$ over $c$ and the irreducible
components of the fiber of $\pi'\circ\mathrm{pr}_{S'}$ over $c$.
It follows from the construction of the smooth relatively connected
quotients $\alpha:\mathcal{C}\rightarrow C$ and $\alpha':\mathcal{C}'\rightarrow C$
that $\Psi$ induces a $C$-isomorphism $\psi:\mathcal{C}\rightarrow\mathcal{C}'$
such that $\psi\circ(\rho\circ\mathrm{pr}_{S})=(\rho'\circ\mathrm{pr}_{S'})\circ\Psi$.
Since the morphisms $\rho\circ\mathrm{pr}_{S}:S\times\mathbb{A}^{1}\rightarrow\mathcal{C}$
and $\rho'\circ\mathrm{pr}_{S'}:S'\times\mathbb{A}^{1}\rightarrow\mathcal{C}'$
are \'etale locally trivial $\mathbb{A}^{2}$-bundles over $\mathcal{C}$
and $\mathcal{C}'$ respectively, the pull-back homomorphisms 
\[
(\rho\circ\mathrm{pr}_{S})^{*}:\mathrm{Pic}(\mathcal{C})\rightarrow\mathrm{Pic}(S\times\mathbb{A}^{1})\quad\textrm{and}\quad(\rho'\circ\mathrm{pr}_{S'})^{*}:\mathrm{Pic}(\mathcal{C}')\rightarrow\mathrm{Pic}(S'\times\mathbb{A}^{1})
\]
are both isomorphisms. Let $T_{S\times\mathbb{A}^{1}/\mathcal{C}}$
and $T_{S'\times\mathbb{A}^{1}/\mathcal{C}'}$ be the relative tangent
bundles of the morphisms $\rho\circ\mathrm{pr}_{S}$ and $\rho'\circ\mathrm{pr}_{S'}$
respectively. By definition of $L_{S/\mathcal{C}}$ and $L_{S'/\mathcal{C}'}$
(see subsection \ref{subsec:L-torsors}), we have 
\[
\Lambda^{2}T_{S\times\mathbb{A}^{1}/\mathcal{C}}\simeq\mathrm{pr}_{S}^{*}T_{S/\mathcal{C}}\simeq(\rho\circ\mathrm{pr}_{S})^{*}L_{S/\mathcal{C}}\;\textrm{and}\;\Lambda^{2}T_{S'\times\mathbb{A}^{1}/\mathcal{C}'}\simeq\mathrm{pr}_{S'}^{*}T_{S'/\mathcal{C}'}\simeq(\rho'\circ\mathrm{pr}_{S'})^{*}L_{S'/\mathcal{C}'}.
\]
Since on the other hand $\Psi$ is an isomorphism and $\psi\circ(\rho\circ\mathrm{pr}_{S})=(\rho'\circ\mathrm{pr}_{S'})\circ\Psi$,
we have
\[
\begin{array}{ccc}
(\rho\circ\mathrm{pr}_{S})^{*}L_{S/\mathcal{C}}\simeq\Lambda^{2}T_{S\times\mathbb{A}^{1}/\mathcal{C}} & \simeq & \Psi^{*}\Lambda^{2}T_{S'\times\mathbb{A}^{1}/\mathcal{C}'}\\
 & \simeq & \Psi^{*}((\rho'\circ\mathrm{pr}_{S'})^{*}L_{S'/\mathcal{C}'})\\
 & \simeq & (\rho\circ\mathrm{pr}_{S})^{*}(\psi^{*}L_{S'/\mathcal{C}'})
\end{array}
\]
from which it follows that $L_{S/\mathcal{C}}\simeq\psi^{*}L_{S'/\mathcal{C}'}$.

Conversely, assume that there exists a $C$-isomorphism $\psi:\mathcal{C}\rightarrow\mathcal{C}'$
such that $\psi^{*}L_{S'/\mathcal{C}'}\simeq L_{S/\mathcal{C}}$.
Letting $S''=S'\times_{\mathcal{C}'}\mathcal{C}$ and $\pi''=\alpha\circ\mathrm{pr}_{\mathcal{C}}:S''\rightarrow C$,
it is enough to constructs a $\mathcal{C}$-isomorphism $S\times\mathbb{A}^{1}\simeq S''\times\mathbb{A}^{1}$.
We can thus assume without loss of generality that $\mathcal{C}=\mathcal{C}'$
and that $\psi=\mathrm{id}_{\mathcal{C}}$. We let $L=L_{S/\mathcal{C}}=L_{S'/\mathcal{C}}$.
Note that since $C$ is affine and the morphisms $\pi:S\rightarrow C$
and $\pi':S'\rightarrow C'$ are affine by definition, the surfaces
$S$ and $S'$ are both affine. Since $C$ is affine, the \'etale
line bundle $p:L_{S/\mathcal{C}}\rightarrow\mathcal{C}$ has a non-zero
section. Indeed, given any rational section $\sigma$ of $L$ we can
find a non-zero regular function $f$ on $C$ which vanishes sufficiently
on the images by $\alpha:\mathcal{C}\rightarrow C$ of the poles of
$\sigma$ so that $s=(\alpha^{*}f)\sigma$ is a regular global section
of $L$ on $\mathcal{C}$. The cokernel $\mathcal{Q}$ of $s$ viewed
as an injective homomorphism $\mathcal{O}_{\mathcal{C}}\rightarrow\mathcal{L}$,
where $\mathcal{L}$ denote the \'etale sheaf of germs of sections
of $L$, is a torsion sheaf on $\mathcal{C}$. It follows that $H_{\mathrm{\acute{e}t}}^{1}(\mathcal{C},\mathcal{Q})=0$.
Considering the long exact sequence of \'etale cohomology associated
to the short exact sequence
\[
0\rightarrow\mathcal{O}_{\mathcal{C}}\stackrel{s}{\rightarrow}\mathcal{L}\rightarrow\mathcal{Q}\rightarrow0,
\]
 of \'etale sheaves on $\mathcal{C}$, we conclude that the homomorphism
\[
H^{1}(s):H_{\mathrm{\acute{e}t}}^{1}(\mathcal{C},\mathbb{G}_{a})=H_{\mathrm{\acute{e}t}}^{1}(\mathcal{C},\mathcal{O}_{\mathcal{C}})\rightarrow H_{\mathrm{\acute{e}t}}^{1}(\mathcal{C},\mathcal{L})=H_{\mathrm{\acute{e}t}}^{1}(\mathcal{C},L)
\]
is surjective. This implies the existence of an \'etale $\mathbb{G}_{a,\mathcal{C}}$-torsor
$\rho_{0}:S_{0}\rightarrow\mathcal{C}$ and a $\mathcal{C}$-morphism
$\xi:S_{0}\rightarrow S$ of torsors which is equivariant for the
homomorphism of group schemes 
\[
\mathrm{Spec}(\mathrm{Sym}^{\cdot}(^{t}s)):\mathbb{G}_{a,\mathcal{C}}=\mathrm{Spec}(\mathrm{Sym}^{\cdot}\mathcal{O}_{\mathcal{C}})\rightarrow L=\mathrm{Spec}(\mathrm{Sym}^{\cdot}\mathcal{L}^{\vee})
\]
induced by $s$ and which has the property that the isomorphism class
of the $L$-torsor $\rho:S\rightarrow\mathcal{C}$ in $H_{\mathrm{\acute{e}t}}^{1}(\mathcal{C},L)$
is the image by $H^{1}(s)$ of the isomorphism class of the \'etale
$\mathbb{G}_{a,\mathcal{C}}$-torsor $\rho_{0}:S_{0}\rightarrow\mathcal{C}$
in $H_{\mathrm{\acute{e}t}}^{1}(\mathcal{C},\mathbb{G}_{a})$. Since
$S$ is affine and $\xi:S_{0}\rightarrow S$ is an affine morphism,
$S_{0}$ is an affine surface. Now consider the fiber products 
\[
W=S\times_{\mathcal{C}}S_{0}\quad\textrm{and}\quad W'=S_{0}\times_{\mathcal{C}}S'.
\]
Since $\rho_{0}:S_{0}\rightarrow\mathcal{C}$ is an \'etale $\mathbb{G}_{a,\mathcal{C}}$-torsor,
the projections $\mathrm{pr}_{S}:W\rightarrow S$ and $\mathrm{pr}_{S'}:W'\rightarrow S'$
are \'etale $\mathbb{G}_{a}$-torsors over $S$ and $S'$ respectively.
Since $S$ and $S'$ are both affine, these torsors are thus trivial
\cite[Proposition 16.5.16]{EGAIV-4}, respectively isomorphic to $S\times\mathbb{A}^{1}$
and $S'\times\mathbb{A}^{1}$ on which $\mathbb{G}_{a}$ acts by translations
on the second factor. On the other hand, the projections $\mathrm{pr}_{S_{0}}:W\rightarrow S_{0}$
and $\mathrm{pr}_{S_{0}}:W'\rightarrow S_{0}$ are both \'etale torsors
under the \'etale line bundle $\rho_{0}^{*}L$ on $S_{0}$. Since
$S_{0}$ is an affine scheme, $\rho_{0}^{*}L$ is a Zariski locally
trivial line bundle and the vanishing of $H_{\mathrm{\acute{e}t}}^{1}(S_{0},\rho_{0}^{*}L)=H^{1}(S_{0},\rho_{0}^{*}L)$
implies that $W$ and $W'$ are both isomorphic to the trivial $\rho_{0}^{*}L$-torsor
over $S_{0}$, that is, to the total space of the line bundle $\rho_{0}^{*}L$
on $S_{0}$. Summing up, we have constructed isomorphisms
\[
S\times\mathbb{A}^{1}\simeq W\simeq\rho^{*}L_{0}\simeq W'\simeq S'\times\mathbb{A}^{1}
\]
 of schemes over $\mathcal{C}$. This completes the proof.
\end{proof}
\begin{rem}
The proof of Theorem \ref{thm:Cylinders} depends in a crucial way
on the existence of an \'etale $\mathbb{G}_{a,\mathcal{C}}$-torsor
$\rho_{0}:S_{0}\rightarrow\mathcal{C}$ with affine total space $S_{0}$.
Such a torsor does not exist in general if the curve $C$ is not affine.
For instance, if $\mathcal{C}=C=\mathbb{P}^{1}$, the vanishing of
$H^{1}(\mathbb{P}^{1},\mathcal{O}_{\mathbb{P}^{1}})$ implies that
the only $\mathbb{G}_{a}$-torsor over $\mathbb{P}^{1}$ is the trivial
one $\mathbb{P}^{1}\times\mathbb{A}^{1}$ on which $\mathbb{G}_{a}$
acts by translations on second factor, whose total space is not affine.
Similiarly, if $\mathcal{C}=C$ is a smooth elliptic curve, then the
total space of the unique non-trivial $\mathbb{G}_{a}$-torsor $\rho_{0}:S_{0}\rightarrow C$
corresponding to the unique non-trivial extension 
\[
0\rightarrow\mathcal{O}_{C}\rightarrow\mathcal{E}\rightarrow\mathcal{O}_{C}\rightarrow0
\]
via the isomorphism $\mathrm{Ext}^{1}(\mathcal{O}_{C},\mathcal{O}_{C})\simeq H^{1}(C,\mathcal{O}_{C})\simeq\mathbb{C}$
is a quasi-projective surface which is not affine.
\end{rem}

\begin{example}
Consider again the smooth affine surfaces $S_{n,2}$ in $\mathbb{A}^{3}$
defined by the equations $x^{n}z=y^{2}-x$, $n\geq2$. By Example
\ref{exa:Mult2-fiber}, each of these surfaces is an \'etale torsor
$\rho_{n,2}:S_{n,2}\rightarrow\mathcal{C}$ under the tangent line
bundle $T_{\mathcal{C}}$ of the multifold algebraic space curve $\alpha:\mathcal{C}\rightarrow\mathbb{A}^{1}$
obtained as the quotient of the affine line with a double origin by
a free $\mu_{2}$-action. One can check that up to composition by
automorphisms of $\mathbb{A}^{1}$, $\pi_{n,2}=\mathrm{pr}_{x}|_{S_{n,2}}:S_{n,2}\rightarrow\mathbb{A}^{1}$
is the unique $\mathbb{A}^{1}$-fibration of affine type on $S_{n,2}$.
Combined with the fact that the \'etale $\mathbb{A}^{1}$-bundles
$\rho_{n,2}:S_{n,2}\rightarrow\mathcal{C}$ are pairwise non-isomorphic,
this implies that the surfaces $S_{n,2}$ are pairwise non-isomorphic
as abstract varieties. On the other hand, it follows from Theorem
\ref{thm:Cylinders} that the cylinders $S_{n,2}\times\mathbb{A}^{1}$,
$n\geq2$, are all isomorphic. We thus recover a particular case of
non-cancellation for so-called \emph{affine pseudo-planes} studied
in \cite{MaMy-05}.
\end{example}

We conclude with the following theorem which provides a complete answer
to the cancellation problem for smooth affine surfaces admitting $\mathbb{A}^{1}$-fibrations
of affine type.
\begin{thm}
\label{thm:Non-Cancel} Let $S$ be a smooth affine surface and let
$\pi:S\rightarrow C$ be an $\mathbb{A}^{1}$-fibration over a smooth
affine curve $C$. Then the following alternative holds:

a) If $\pi:S\rightarrow C$ is isomorphic to the structure morphism
of a line bundle over $C$ then every smooth affine surface $S'$
such that $S'\times\mathbb{A}^{1}\simeq S\times\mathbb{A}^{1}$ is
isomorphic to $S$.

b) Otherwise, there exists a smooth affine $\mathbb{A}^{1}$-fibered
surface $S'$ non-isomorphic to $S$ such that $S\times\mathbb{A}^{1}$
is isomorphic to $S'\times\mathbb{A}^{1}$.
\end{thm}

\begin{proof}[Sketch of proof]
 Assume that $\pi:S\rightarrow C$ is a line bundle, let $S'$ be
a smooth affine surface and let $\Psi:S'\times\mathbb{A}^{1}\stackrel{\simeq}{\rightarrow}S\times\mathbb{A}^{1}$
be an isomorphism of abstract algebraic varieties. If $C\simeq\mathbb{A}^{1}$
then $\pi:S\rightarrow\mathbb{A}^{1}$ is a trivial line bundle so
that $S\simeq\mathbb{A}^{2}$. The assertion then follows from \cite{MS80}.
Otherwise, if $C\not\simeq\mathbb{A}^{1}$ then since every morphism
$\mathbb{A}^{1}\rightarrow C$ is constant, it follows that the composition
of $\Psi$ with the projection $\pi\circ\mathrm{pr}_{S}:S\times\mathbb{A}^{1}\rightarrow C$
descends to a unique morphism $\pi':S'\rightarrow C$ such that $(\pi\circ\mathrm{pr}_{S})\circ\Psi=\pi'\circ\mathrm{pr}_{S'}$.
Since $\pi\circ\mathrm{pr}_{S}$ is a Zariski locally trivial $\mathbb{A}^{2}$-bundle,
it follows that the same holds for $\pi'\circ\mathrm{pr}_{S'}$. This
implies in turn that $\pi':S'\rightarrow C$ is an $\mathbb{A}^{1}$-fibration
without degenerate fibers, hence is a line bundle as $C$ is affine.
Arguing as in the proof of Theorem\ref{thm:Cylinders}, we conclude
that $\pi:S\rightarrow C$ and $\pi':S'\rightarrow C$ are isomorphic
line bundles over $C$, hence that $S\simeq S'$.

Now assume that $\pi:S\rightarrow C$ is not isomorphic to the structure
morphism of a line bundle over $C$. Since $C$ is affine, it follows
that $\pi:S\rightarrow C$ has at least a degenerate fiber. Let $\alpha:\mathcal{C}\rightarrow C$
be the smooth relatively connected quotient of $\pi:S\rightarrow C$
and let $\rho:S\rightarrow\mathcal{C}$ be the associated \'etale
$L_{S/\mathcal{C}}$-torsor. Since $\pi:S\rightarrow C$ has a degenerate
fiber, it follows from the construction of the smooth relatively connected
quotient that $\mathcal{C}$ is a non-separated algebraic space or
scheme and that $\alpha$ is not an isomorphism. Since $S$ is affine,
hence is a separated scheme, this implies in particular that $\rho:S\rightarrow\mathcal{C}$
is a non-trivial \'etale $L_{S/\mathcal{C}}$-torsor whose isomorphism
class in $H_{\mathrm{\acute{e}t}}^{1}(\mathcal{C},L_{S/\mathcal{C}})$
is thus a non-zero element.

Let $\mathcal{L}$ be the \'etale invertible sheaf of germs of section
of $L_{S/\mathcal{C}}$ and let $f$ be a non-zero regular function
on $C$. The multiplication by $\alpha^{*}f$ defines an injective
homomorphism $s:\mathcal{L}\rightarrow\mathcal{L}$ which determines
in turn a non-zero homomorphism of group schemes $\zeta:L_{S/\mathcal{C}}\rightarrow L_{S/\mathcal{C}}$.
Arguing as in the proof of Theorem \ref{thm:Cylinders}, we can find
an \'etale $L_{S/\mathcal{C}}$-torsor $\rho':S'\rightarrow\mathcal{C}$
with affine total space and an $s$-equivariant $\mathcal{C}$-morphism
$\xi:S'\rightarrow S$ of torsors with the property that the isomorphism
class of $\rho:S\rightarrow\mathcal{C}$ in $H_{\mathrm{\acute{e}t}}^{1}(\mathcal{C},L_{S/\mathcal{C}})$
is the image by $H^{1}(s)$ of the isomorphism class of $\rho':S'\rightarrow\mathcal{C}$
in $H_{\mathrm{\acute{e}t}}^{1}(\mathcal{C},L_{S/\mathcal{C}})$.

We claim that by choosing $f\in\Gamma(C,\mathcal{O}_{C})$ appropriately,
we can ensure simultaneously that on the one hand the morphism $\pi'=\alpha\circ\rho':S'\rightarrow C$
is the unique $\mathbb{A}^{1}$-fibration of affine type on $S'$
and that on the other hand, for every automorphism $\psi$ of $\mathcal{C}$,
the $\mathbb{A}^{1}$-bundles $\rho:S\rightarrow\mathcal{C}$ and
$\mathrm{pr}_{2}:S'\times_{\rho',\mathcal{C},\psi}\mathcal{C}\rightarrow\mathcal{C}$
are not isomorphic. Indeed, the first condition is automatically satisfied
unless $C$ is isomorphic to the affine line $\mathbb{A}^{1}$. In
the case where $C=\mathbb{A}^{1}$, it is enough to choose a regular
function $f$ which vanishes sufficiently at the points of $C$ over
which the fibers of $\pi:S\rightarrow C$ are degenerate to ensure
the existence of an SNC-minimal projective completion $\overline{S}'$
of $S'$ whose boundary divisor is not a chain, a property which implies
that $\pi':S'\rightarrow C$ is the unique $\mathbb{A}^{1}$-fibration
of affine type on $S'$ (see e.g. \cite[Théorème 1.8]{Be83} or \cite[Theorem 2.16]{Dub04}).
Similarly, choosing $f$ so that it vanishes at the points of $C$
over which the fibers of $\pi:S\rightarrow C$ are degenerate is enough
to guarantee that the images of the isomorphism classes of $\rho:S\rightarrow\mathcal{C}$
and $\rho':S'\rightarrow\mathcal{C}$ in $\mathbb{P}H_{\mathrm{\acute{e}t}}^{1}(\mathcal{C},L_{S/\mathcal{C}})$
do not belong the same orbit of the action of the automorphism group
of $\mathcal{C}$.

To conclude, one checks that for a regular function $f$ satisfying
the two properties above, the surfaces $S$ and $S'$ are not isomorphic
as abstract algebraic varieties. On the other hand, since $\rho:S\rightarrow\mathcal{C}$
and $\rho':S'\rightarrow\mathcal{C}$ are both \'etale $L_{S/\mathcal{C}}$-torsors,
we deduce from Theorem \ref{thm:Cylinders} that $S\times\mathbb{A}^{1}$
is isomorphic to $S'\times\mathbb{A}^{1}$.
\end{proof}
\begin{example}
(Danielewski counter-example \cite{Dan89} revisited). Let $S_{0}$
be the smooth surface in $\mathbb{A}^{3}$ defined by the equation
$xz=y^{2}-1$. The restriction of the projection $\mathrm{pr}_{x}$
on $S$ defines a smooth $\mathbb{A}^{1}$-fibration $\pi_{0}:S_{0}\rightarrow\mathbb{A}^{1}$
restricting to a trivial $\mathbb{A}^{1}$-bundle over $\mathbb{A}^{1}\setminus\{0\}$
and whose fiber $\pi^{-1}(0)$ is reduced, consisting of two disjoint
copies $\{x=y\pm1=0\}$ of the affine line $\mathbb{A}^{1}=\mathrm{Spec}(\mathbb{C}[z])$.
The smooth relatively connected quotient of $\pi_{0}:S_{0}\rightarrow\mathbb{A}^{1}$
is thus isomorphic to the affine line with a double origin $\alpha:\tilde{\mathbb{A}}^{1}\rightarrow\mathbb{A}^{1}$.
On checks using local trivialization as in the Example in subsection
\ref{subsec:Toy-Example} that the induced morphism $\rho_{0}:S_{0}\rightarrow\tilde{\mathbb{A}}^{1}$
is a $\mathbb{G}_{a,\tilde{\mathbb{A}}^{1}}$-torsor, whose isomorphism
class in $H^{1}(\tilde{\mathbb{A}}^{1},\mathbb{G}_{a})=H^{1}(\tilde{\mathbb{A}}^{1},\mathcal{O}_{\tilde{\mathbb{A}}_{1}})$
is equal to the class of the \v{C}ech cocycle $g_{0}=2x^{-1}\in C^{1}(\mathcal{U},\mathcal{O}_{\tilde{\mathbb{A}}^{1}})\simeq\mathbb{C}[x^{\pm1}]$
for the natural open cover $\mathcal{U}$ of $\tilde{\mathbb{A}}^{1}$
by two copies of $\mathbb{A}^{1}=\mathrm{Spec}(\mathbb{C}[x])$. Let
$s_{n}:\mathcal{O}_{\tilde{\mathbb{A}}^{1}}\rightarrow\mathcal{O}_{\tilde{\mathbb{A}}^{1}}$
be the injective homomorphism given by the multiplication by $x^{n}$
and let 
\[
\zeta:\mathbb{G}_{a,\tilde{\mathbb{A}}^{1}}=\mathrm{\mathrm{Spec}}(\mathcal{O}_{\tilde{\mathbb{A}}^{1}}[t])\rightarrow\mathrm{\mathrm{Spec}}(\mathcal{O}_{\tilde{\mathbb{A}}^{1}}[t])=\mathbb{G}_{a,\tilde{\mathbb{A}}^{1}},\;t\mapsto x^{n}t
\]
 be the corresponding homomorphism of group schemes over $\tilde{\mathbb{A}}^{1}$.
The induced homomorphism 
\[
H^{1}(s_{n}):H^{1}(\tilde{\mathbb{A}}^{1},\mathcal{O}_{\tilde{\mathbb{A}}^{1}})\rightarrow H^{1}(\tilde{\mathbb{A}}^{1},\mathcal{O}_{\tilde{\mathbb{A}}^{1}})
\]
maps the class in $H^{1}(\tilde{\mathbb{A}}^{1},\mathcal{O}_{\tilde{\mathbb{A}}^{1}})$
of the \v{C}ech cocycle $g_{n}=2x^{-n-1}\in C^{1}(\mathcal{U},\mathcal{O}_{\tilde{\mathbb{A}}^{1}})$
onto the isomorphism class of the $\mathbb{G}_{a,\tilde{\mathbb{A}}^{1}}$-torsor
$\rho_{0}:S_{0}\rightarrow\tilde{\mathbb{A}}^{1}$. It follows that
there exists a $\mathbb{G}_{a,\tilde{\mathbb{A}}^{1}}$-torsor $\rho_{n}:S_{n}\rightarrow\tilde{\mathbb{A}}^{1}$
whose isomorphism class in $H^{1}(\tilde{\mathbb{A}}^{1},\mathcal{O}_{\tilde{\mathbb{A}}^{1}})$
is equal to the class of the cocyle $g_{n}$ and a $\zeta$-equivariant
morphism of $\mathbb{G}_{a,\tilde{\mathbb{A}}^{1}}$-torsors $\xi:S_{n}\rightarrow S_{0}$.

The desired $\mathbb{G}_{a,\tilde{\mathbb{A}}^{1}}$-torsor $\rho_{n}:S_{n}\rightarrow\tilde{\mathbb{A}}^{1}$
is given for instance by the smooth affine surface $S_{n}$ in $\mathbb{A}^{3}$
with equation $x^{n+1}z=y^{2}-1$ endowed with the factorization $\rho_{n}:S_{n}\rightarrow\tilde{\mathbb{A}}^{1}$
of the smooth $\mathbb{A}^{1}$-fibration $\pi_{n}:S_{n}\rightarrow\mathbb{A}^{1}$
induced by the restriction of the projection $\mathrm{pr}_{x}$. A
corresponding $s_{n}$-equivariant morphism is simply the birational
morphism 
\[
\xi:S_{n}\rightarrow S_{0},\quad(x,y,z)\mapsto(x,y,x^{n}z).
\]

By Theorem \ref{thm:Cylinders}, the cylinders $S_{n}\times\mathbb{A}^{1}$,
$n\geq0$, are all isomorphic. Clearly, the surface $S_{0}$ admits
a second $\mathbb{A}^{1}$-fibration of affine type given by the restriction
of the projection $\mathrm{pr}_{z}$. On the other hand, for every
$n\geq2$, $\pi_{n}:S_{n}\rightarrow\mathbb{A}^{1}$ is the unique
$\mathbb{A}^{1}$-fibration of affine type on $S_{n}$ up to composition
by automorphisms of $\mathbb{A}^{1}$ (see e.g. \cite{ML01}). It
follows that for every $n\geq2$, $S_{n}$ is not isomorphic to $S_{0}$.
Actually, the surfaces $S_{n}$, $n\geq0$, are even pairwise non
isomorphic \cite{Dan89,ML01}.
\end{example}

\bibliographystyle{amsplain}

\end{document}